\documentclass[12pt,letterpaper]{article}
\usepackage{amsmath,amsfonts,amsthm,amssymb}

\swapnumbers 

\newtheorem{lemma}{Lemma}[section]
\newtheorem{corollary}[lemma]{Corollary}
\newtheorem{theorem}[lemma]{Theorem}

\theoremstyle{definition} 
\newtheorem{definition}[lemma]{Definition}
\newtheorem{remark}[lemma]{Remark}

\newtheorem{remarks}[lemma]{Remarks}

\newcommand{\Nat}{{\mathbb N}}

\newcommand\reals{{\mathbb R}}
\newcommand\ratnls{{\mathbb Q}}

\newcommand{\dg}{\sp{\text{\rm o}}}

\begin{document}

\title{Spectral order on a synaptic algebra}

\author{David J. Foulis{\footnote{Emeritus Professor, Department of
Mathematics and Statistics, University of Massachusetts, Amherst,
MA; Postal Address: 1 Sutton Court, Amherst, MA 01002, USA;
foulis@math.umass.edu.}}\hspace{.05 in},  Sylvia
Pulmannov\'{a}{\footnote{ Mathematical Institute, Slovak Academy of
Sciences, \v Stef\'anikova 49, SK-814 73 Bratislava, Slovakia;
pulmann@mat.savba.sk. The second author was supported by the 
Research and Development Support Agency under the contract 
APVV-16-0073 and grant VEGA 2/0069/16. }}}

\date{}

\maketitle

\begin{abstract}
\noindent We define and study an alternative partial order, called the spectral order,
on a synaptic algebra---a generalization of the self-adjoint part of a von Neumann
algebra. We prove that if the synaptic algebra  $A$ is norm complete (a Banach synaptic algebra),
then under the spectral order, $A$ is Dedekind $\sigma$-complete lattice, and the corresponding
effect algebra $E$ is a $\sigma$-complete lattice. Moreover, $E$ can be organized into a Brouwer-Zadeh
algebra in both the  usual (synaptic) and spectral ordering; and if $A$ is Banach, then $E$ is a
Brouwer-Zadeh lattice in the spectral ordering. If $A$ is of finite type, then De\,Morgan laws hold
on $E$ in both the synaptic and spectral ordering.
\end{abstract}

\section{Introduction}

In this paper, `iff' abbreviates `if and only if,' the notation := means `equals by definition,'
$\reals$ is the ordered field of real numbers, and $\Nat:=\{1,2,3,...\}$ is the ordered set of
natural numbers.

We shall often refer to the partial order relation $\leq$ on a partially ordered set (poset)
${\mathcal P}$ simply as an \emph{order} or an \emph{ordering}. The poset ${\mathcal P}$ is
\emph{lattice ordered}, or simply a \emph{lattice} iff every pair of elements $a,b\in{\mathcal P}$
has an infimum $a\wedge b$ (a greatest lower bound) and a supremum $a\vee b$ (a least upper bound).
An existing infimum (respectively, supremum) of a family $(a\sb{\gamma})\sb{\gamma\in\Gamma}$ in
${\mathcal P}$ is written as $\bigwedge\sb{\gamma\in\Gamma}a\sb{\gamma}$ (respectively, as $\bigvee
\sb{\gamma\in\Gamma}a\sb{\gamma}$).

We denote the von Neumann algebra of all bounded linear operators on a Hilbert space ${\mathfrak H}$
by ${\mathcal B}({\mathfrak H})$ and we denote the self-adjoint part of ${\mathcal B}({\mathfrak H})$
by ${\mathcal B}\sp{sa}({\mathfrak H})$. If $\langle\cdot,\cdot\rangle$ is the inner product on
${\mathfrak H}$, then the usual order $\leq$ for ${\mathcal B}\sp{sa}({\mathfrak H})$ is defined
for $A,B\in{\mathcal B}\sp{sa}({\mathfrak H})$ by $A\leq B$ iff $\langle Ax,x\rangle\leq\langle Bx,
x\rangle$ for all $x\in{\mathfrak H}$. Note that $A\leq B$ iff the spectrum of $B-A$ is contained in
$\{\alpha\in\reals:0\leq\alpha\}$. Following S. Gudder \cite{Gudlog}, we shall refer to $\leq$ as
the \emph{numerical order} on ${\mathcal B}\sp{sa}({\mathfrak H})$ and on subalgebras thereof. The
``unit interval" ${\mathcal E}({\mathfrak H}):=\{E\in{\mathcal B}\sp{sa}({\mathfrak H}):0\leq E\leq 1\}$,
under the partial binary operation $\oplus$ defined by $E\oplus F:=E+F$ iff $E+F\leq 1$, is called the
\emph{standard effect algebra}, and it is the prototype for the more general notion of an \emph{effect
algebra} \cite{FoBe}.

Let ${\mathcal R}$ be a von Neumann algebra and let ${\mathcal R}\sp{sa}$ be the self-adjoint part
of ${\mathcal R}$ with the numerical order. The system ${\mathcal E}({\mathcal R}):=\{E\in{\mathcal R}
\sp{sa}:0\leq E\leq 1\}$ with the partial binary operation $\oplus$ defined by $E\oplus F=E+F$ iff
$E+F\leq 1$ is an effect algebra that generalizes the standard effect algebra.

By a well-known theorem of S. Sherman \cite{Sh}, ${\mathcal R}\sp{sa}$ is a lattice iff ${\mathcal R}$
is commutative. Recall that ${\mathcal R}$ is a \emph{factor} iff its center consists only of scalar
multiples of the identity, and ${\mathcal R}\sp{sa}$ is an \emph{antilattice} iff, for all $a,b\in
{\mathcal R}\sp{sa}$, the infimum $a\wedge b$ exists in ${\mathcal R}\sp{sa}$ iff $a\leq b$ or $b
\leq a$. By \emph{Kadison's antilattice theorem} \cite{Kad}, ${\mathcal R}$ is a factor iff ${\mathcal R}
\sp{sa}$ is an antilattice. Thus, since ${\mathcal B}({\mathfrak H})$ is a factor, it follows that
${\mathcal B}\sp{sa}({\mathfrak H})$ is an antilattice under the numerical order.

Kadison's antilattice theorem has engendered considerable research on possible alternative orders
for ${\mathcal B}\sp{sa}({\mathfrak H})$ and related operator algebras that, optimally, would
result in a lattice, or at least in a poset with some lattice-like properties. Three (actually two)
of these alternative orders are as follows:

(1) Drazin's \emph{star order} $\leq\sb{\ast}$ \cite{Dr}, which was originally defined for so-called
proper involution rings, makes sense for a von Neumann algebra ${\mathcal R}$, and so does its
restriction to the self-adjoint part ${\mathcal R}\sp{sa}$ of ${\mathcal R}$. The star order on
${\mathcal R}\sp{sa}$ has various equivalent characterizations, one of which is as follows.
For $A,B\in{\mathcal R}\sp{sa}$, $A\leq\sb{\ast}B$ iff there exists $C\in{\mathcal R}\sp{sa}$ such
that $AC=0$ and $A+C=B$.  A number of authors, e.g., \cite{BoHa, Ci, PuVi, WJ} have studied the star
order for various operator algebras and related structures.

(2) S.P. Gudder \cite{Gudlog} has formulated a so-called \emph{logical order} $\preceq$ for
${\mathcal B}\sp{sa}({\mathfrak H})$ as follows: If $A\in{\mathcal B}\sp{sa}({\mathfrak H})$,
${\mathcal F}(\reals)$ is the $\sigma$-field of real Borel sets, and ${\mathcal P}({\mathfrak H})$ is
the lattice under the numerical order of all projections $P=P\sp{2}\in{\mathcal B}\sp{sa}({\mathfrak H})$,
denote the real observable corresponding to $A$ by $P\sp{A}\colon{\mathcal F}(\reals)\to{\mathcal P}
({\mathfrak H})$. Then for $A,B\in{\mathcal B}\sp{sa}({\mathfrak H})$, $A\preceq B$ iff
$P\sp{A}(\Delta)\leq P\sp{B}(\Delta)$ for all $\Delta\in{\mathcal F}(\reals)$ with $0\notin\Delta$.
By \cite[Theorem 4.6]{Gudlog}, $A\preceq B\Leftrightarrow A\leq\sb{\ast}B$. It is also shown in
\cite{Gudlog} that, under the logical order, ${\mathcal B}\sp{sa}({\mathfrak H})$ is a near lattice, i.e.,
if for $A,B\in{\mathcal B}\sp{sa}({\mathfrak H})$ there is $C\in{\mathcal B}\sp{sa}({\mathfrak H})$ with
$A,B\preceq C$, then both the supremum and infimum of $A$ and $B$ with respect to $\preceq$ exist. By
\cite[Theorem 4.17]{Gudlog}, if  $H$ is finite dimensional, then the logic-order infimum of any two elements
in ${\mathcal B}\sp{sa}({\mathfrak H})$ exists. This result was extended also to infinite dimensional
${\mathfrak H}$ in \cite[Corollary 4.6]{PuVi} in the sense that, in ${\mathcal B}\sp{sa}({\mathfrak H})$,
the logic-order infimum of any nonempty subset exists and the logic-order supremum of any nonempty subset
that is bounded above exists. In \cite{WJ}, explicit forms are found for logical infima and suprema.

Also, under the logical (or star) order, ${\mathcal B}\sp{sa}({\mathfrak H})$ forms a so-called \emph
{generalized $\sigma$-orthoalgebra} \cite[\S 2]{Gudlog} under the partial operation $A\oplus B:=A+B$
iff $AB=0$, and $\preceq$ is the corresponding induced order. The logic order has been studied in many
further papers, e.g. \cite{BoHa, Ci, PuVi}, and the notion of logical (or star) order has been extended
to so-called \emph{generalized Hermitian algebras} \cite{FPgh} by Y. Li and X. Xu \cite{LiXu}.

(3) M.P. Olson \cite{Ols} introduced the \emph{spectral order} $\leq\sb{s}$ for the self-adjoint part
${\mathcal R}\sp{sa}$ of a von Neumann algebra. See Definition \ref{de:specord} below. For ${\mathcal R}
\sp{sa}$, the spectral order agrees with the numerical order on projections (idempotent elements $P=P\sp{2}$)
and on commutative subalgebras. Under the spectral order, ${\mathcal R}\sp{sa}$ forms a Dedekind complete lattice.
The set of effects ${\mathcal E}({\mathcal R})$ under the spectral order is studied in \cite{GLP}, and in
\cite{CaHa} it is shown that it is a complete Brouwer-Zadeh lattice \cite{dG} in which both De\,Morgan laws
are satisfied for the Brouwer complement iff $\mathcal R$ is of finite type. See also \cite{H} for related results
on Jordan algebras and \cite{HT1, HT2} on AW*-algebras.

Our purpose in this paper, as indicated by its title, is to study the spectral order on a so-called \emph{synaptic
algebra} (abbreviated SA) \cite{Fsa, FPproj, FPreg, FPtd, FPsym, FJPtp, FJPep, FPcom, FPVecLat, FPantilat, FPba, FJPls,
Pid}. The formulation of a synaptic algebra \cite{Fsa} was motivated by an effort to define, by means of a few simple
and physically plausible axioms, an algebraic structure suitable for the mathematical description of a quantum mechanical
system. A synaptic algebra (from the Greek `sunaptein,' meaning to join together), unites the notions of an order-unit
normed space \cite{Alf}, a special real unital Jordan algebra \cite{McC}, a convex sharply dominating effect algebra
\cite{Gudsd, GBBP}, and an orthomodular lattice \cite{Be,Ka}. Synaptic algebras generalize the self-adjoint parts of Rickart
C$\sp{\ast}$-algebras, AW$\sp{\ast}$-algebras, and von Neumann algebras, as well as JB-algebras, JBW$\sp{\ast}$-algebras,
spin factors \cite{Top}, and OJ-algebras \cite{Sa}.  Also, generalized Hermitian algebras \cite{FPgh, FPspin, FPreg}, are
special cases of SAs. Axioms SA1--SA8 for an SA and some of the basic properties of an SA can be found in Section
\ref{sc:SAs} below.

A synaptic algebra is lattice ordered if and only if it is commutative \cite[Theorem 5.6]{FPVecLat}. An analogue of
Kadison's antilattice theorem \cite{Kad} is proved for SAs in \cite{FPantilat}.

Let $A$ be a synaptic algebra, let $E:=\{e\in A:0\leq e\leq 1\}$ be the set of effects in $A$, and let $P:=\{p\in A:
p=p\sp{2}\}$ be the orthomodular lattice (OML) of projections in $A$. By \cite[Theorem 5.3]{FPba}, $A$ is Banach
(i.e., norm complete) iff it is isomorphic to the self-adjoint part of a Rickart C$\sp{\ast}$-algebra. If $A$ is
Banach, then under the spectral order, $A$ is a Dedekind $\sigma$-complete lattice and $E$ is a $\sigma$-complete
lattice (Theorem \ref{th:Sigmalat} and Corollary \ref{co:Sigmalat} below); moreover (by a similar proof), if in
addition $P$ is a complete OML, then $A$ is a Dedekind complete lattice and $E$ is a complete lattice. We show that $E$
can be organized into a Brouwer-Zadeh effect algebra in both the synaptic and the spectral orderings; furthermore, if $A$
is Banach, then $E$ is a Brouwer-Zadeh lattice in the spectral ordering (Theorem \ref{th:BZ} below). As a consequence of
Theorem \ref{le:caha} below, if  $P$ is a complete OML and $A$ is of finite type, then under both the  synaptic and spectral
order, the Brouwer complementation on $E$ satisfies the De\,Morgan laws.

We note that, working with synaptic algebras, we cannot make use of the rich Hilbert space structure, and so have to apply alternative, usually more algebraic or order-theoretic, methods. Some of our results overlap with the results in \cite{HT1, HT2} obtained for AW*-algebras, so that they also hold for Banach synaptic algebras with a complete projection lattice. We have shown that all the results can be also obtained from the axioms of synaptic algebras and properties derived from them. Moreover, some of the results hold in more general classes of synaptic algebras.

\section{Synaptic algebras} \label{sc:SAs}

In this section, we introduce the definition of a synaptic algebra \cite{Fsa} and recall some of its properties.

Let $R$ be a linear associative algebra with unit element $1$ over the real or complex numbers and let $A$ be a real
linear subspace of $R$. (To help fix ideas, the reader can think of $R$ as a von Neumann algebra and of $A$ as the
self-adjoint part of $R$.) Let $a,b\in A$. We understand that the product $ab$ is calculated in $R$, and may or may not
belong to $A$. We write $aCb$ iff $ab=ba$. The set $C(a):=\{b\in A: aCb\}$ is the \emph{commutant} of $a$. If
$B\subseteq A$, then $C(B):=\bigcap_{b\in B}C(b)$ is the \emph{commutant} of $B$, $CC(B):=C(C(B))$ is the
\emph{bicommutant} of $B$, and $CC(a):=CC(\{a\})$ is the \emph{bicommutant} of $a$. An element $p\in A$ with $p=p^2$
is called a \emph{projection}, and we denote the set of projections in $A$ by $P$.

The real linear space $A$ is a \emph{synaptic algebra} with \emph{enveloping algebra} $R$ iff the following
conditions are satisfied:
\begin{enumerate}
\item[SA1] $A$ is partially ordered Archimedean real linear space with positive cone $A^+:=\{a\in A:0\leq a\}$,
$1\in A^+$ is an order unit in $A$, and $\|\cdot\|$ is the corresponding order-unit norm on $A$ \cite[pp. 67-69]{Alf}.
We shall assume that $1\not=0$, which enables us, as usual, to identify each real number $\alpha\in\reals$ with the
element $\alpha1\in A$.

\item[SA2] If $a\in A$, then $a^2\in A^+$.

\item[SA3] If $a,b\in A^+$, then $aba\in A^+$.

\item[SA4] If $a\in A$ and $b\in A^+$, then $aba=0\, \implies\, ab=ba=0$.

\item[SA5] If $a\in A^+$, there exists $b\in A^+\bigcap CC(a)$ such that $b^2=a$.

\item[SA6] If $a\in A$, there exits $p\in A$ such that $p=p^2$ and, for all $b\in A$, $ab=0 \Leftrightarrow pb=0$.

\item[SA7] If $1\leq a$ there exists $b\in A$ such that $ab=ba=1$.

\item[SA8] If $a,b\in A$, $a_1\leq a_2\leq \cdots $ is an ascending sequence of pairwise commuting elements of $C(b)$
and $\lim_{n\to \infty} \|a-a_n\|=0$, then $a\in C(b)$.
\end{enumerate}

Let us  briefly recall some  consequences of the axioms SA1-SA8. As a consequence of SA1: (i)  $A$ is a partially
ordered real linear space under $\leq$, called the \emph{synaptic order}; (ii) $0 < 1$ and $1$ is an order unit in
$A$, i.e., for each $a\in A$ there exists $n\in {\mathbb N}$ such that $a\leq n1$; (iii) $A$ is Archimedean, i.e.,
if $a,b\in A$ and $na\leq b$ for all $n\in{\mathbb N}$, then $a\leq 0$. The order-unit norm $\|\cdot\|$ on $A$,
defined by
\[
 \|a\|:=\inf\{0<\lambda\in\reals:-\lambda\leq a\leq\lambda\},
\]
is related to the synaptic order by the following properties: For $a,b\in A$,
\[
-\|a\|\leq a\leq \|a\|, \mbox{and if} \, -b\leq a\leq b, \mbox{then}  \, \|a\|\leq \|b\|.
\]

As a consequence of SA2, $A$ is closed under squaring, whence it forms a special Jordan algebra under the Jordan
product $a\odot b:=\frac{1}{2}(ab+ba)=\frac{1}{2}((a+b)^2-a^2-b^2)\in A$ for $a,b\in A$.

If $a,b\in A$, then $aba=2a\odot(a\odot b)-(a\odot a)\odot b\in A$, hence we may define the \emph{quadratic
mapping} $b\mapsto aba\in A$. For each $a\in A$, the quadratic mapping $b\mapsto aba$, $b\in A$, is obviously
linear and by \cite[Theorem 4.2]{Fsa} it is order preserving, which yields a stronger version of SA3.

Putting $b=1$ in SA4, we find that for $a\in A$, $a\sp{2}=0\Rightarrow a=0$.

By SA5, if $0\leq a\in A$, there exists $0\leq b\in CC(a)$ with $b^2=a$; by \cite[Theorem 2.2]{Fsa}, $b$ is
uniquely determined by $a$; and we define $a^{1/2}:=b$. The \emph{absolute value}  $|a|$ of $a\in A$ is defined
by $|a|:=(a^2)^{1/2}$ \cite[Definition 3.1]{Fsa}. We also define the positive and negative parts of $a\in A$ by
\[
a\sp{+}:=\frac12(|a|+a),\ a\sp{-}:=\frac12(|a|-a),
\]
and we have $0\leq |a|,\, a^+,\, a^- \in CC(a)$, $a=a^+-a^-$, $|a|=a^++a^-$, and $a^+a^-=a^-a^+=0$.

If $a\in A$, then according to SA6, there is a projection $p\in P$ such that, for all $b\in A$, $ab=0  \,
\Leftrightarrow \, pb=0$. The projection $p$ is uniquely determined by $a$  \cite[Theorem 2.7]{Fsa}, and we
define the \emph{carrier} $a\dg$ of $a$ by $a\dg:=p$. (Some authors would refer to $a\dg$ as the \emph
{support} of $a$.) Then, for $a,b\in A$, $ab=0\Leftrightarrow a\dg b=0\Leftrightarrow a\dg b\dg=0
\Leftrightarrow b\dg a\dg=0\Leftrightarrow ba\dg=0\Leftrightarrow ba=0$, $a\dg\in CC(a)$, and $a\dg$ is
the smallest projection $q\in P$ such that $a=aq$ (or equivalently, such that $a=qa$).

By SA7, if $1\leq a$, then $a$ is invertible, and we denote its inverse by $a^{-1}$. We note that if $a$ is
invertible, then $a^{-1}\in CC(a)$.

In the presence of conditions SA1-SA7, condition SA8 is equivalent to the condition that, for every $a\in A$,
the commutant $C(a)$ is norm closed in $A$ \cite[Theorem 9.11]{Fsa}. Hence, for $B\subseteq A$, both $C(B)$
and $CC(B)$ are norm closed in $A$.

Under the partial order inherited from $A$, the set $P$ of projections in $A$ forms an orthomodular lattice (OML) \cite{Be, Ka, PP} with the smallest element $0$, largest element  $1$ and orthocomplementation $p\mapsto
p^{\perp}:=1-p$ \cite[\S 5]{Fsa}. Let $p,q\in P$. It turns out that $p\leq q\,\Leftrightarrow\ p=pq\,
\Leftrightarrow\, p=qp$ and $p\leq q^{\perp}\,\Leftrightarrow\, pq=0\ \Leftrightarrow\, p+q\in P$. We denote by
$p\vee q$ and $p\wedge q$ the supremum and the infimum, respectively, of $p$ and $q$ in $P$. It is not difficult
to show that, if $pCq$, then $p\vee q=p+q-pq$ and $p\wedge q=pq$.

A subset $S\subseteq A$ is a \emph{sub-synaptic algebra} of $A$ iff it is a linear subspace of $A$, $1\in S$,
and $S$ is closed under the formation of squares, square roots, carriers, and inverses. A sub-synaptic algebra
of $A$ is a synaptic algebra in its own right under the restrictions to $S$ of the synaptic order and the
operations on $A$. Let $B\subseteq A$. Then $C(B)$ is a sub-synaptic algebra of $A$. The set $B$ is \emph
{commutative} iff $aCb$ for all $a,b\in B$, i.e., iff $B\subseteq C(B)$. If $B$ is commutative, then $CC(B)$
is a commutative sub-synaptic algebra of $A$. A maximal commutative subset of $A$ is called a \emph{C-block}.
Every C-block is a commutative sub-synaptic algebra of $A$. By Zorn's lemma, every commutative subset of $A$
can be enlarged to a C-block and, as every singleton subset $\{a\}$ of $A$ is commutative, $A$ is covered
by its own C-blocks.

The synaptic algebra $A$ is commutative iff the OML $P$ of projections in $A$ is a Boolean algebra. For a
commutative SA $A$, we have the following functional representation theorem \cite[Theorem 4.1]{FPproj}: Let $X$
denote the Stone space of the Boolean algebra $P$ and let $C(X,\reals)$ be the partially ordered Banach algebra
under pointwise order and operations of all continuous real-valued functions on $X$. Then there is a norm-dense
subset $F\subseteq C(X,{\mathbb R})$ and a synaptic isomorphism $\Psi\colon A\to F$ such that the restriction of
$\Psi$ to $P$ is a Boolean isomorphism of $P$ onto the set of all characteristic functions (indicator functions)
of clopen (both closed and open) subsets of $X$. Moreover, if $A$ is norm complete, then $X$ is basically
disconnected and $F=C(X,{\mathbb R})$ \cite{FPba}.

Elements of the set $E:=\{e\in A: 0\leq e\leq 1\}$ are called \emph{effects}, and under the convex structure
inherited from the linear structure of $A$, $E$ forms a convex effect algebra $(E;0,1,\oplus)$, where, for
$e,f\in E$, $e\oplus f$ is defined iff $e+f\leq 1$, in which case $e\oplus f:=e+f$ \cite{GBBP}. Clearly $P
\subseteq E$ and it turns out that the infimum and supremum $p\wedge q$ and $p\vee q$ of $p$ and $q$ in
$P$ are also the infimum and supremum of $p$ and $q$ in $E$. Therefore, if $e,f\in E$, no confusion will result
if we denote an existing infimum or supremum of $e$ and $f$ in $E$ by $e\wedge f$ or by $e\vee f$. If $e\in E$,
then $e\sp{\perp}:=1-e\in E$ is called the \emph{orthosupplement} of $e$, and $e$ is said to be \emph{sharp}
iff the infimum $e\wedge e\sp{\perp}$ in $E$ exists and equals $0$. It can be shown that, if $e\in E$, then
$e\in P$ iff $e$ is an extreme point of $E$ iff $e$ is sharp.  Moreover, if $p\in P$, $e\in E$, then $p\leq e
\,\Leftrightarrow\, p=pe\,\Leftrightarrow\,p=ep$; similarly, $e\leq p\,\Leftrightarrow\, e=pe\,\Leftrightarrow
\, e=ep$; and if $eCp$, then $ep=e\wedge p=pe$ \cite[Theorems 2.5 and 2.7]{FJPep}. If $e\in E$, then $e\dg\in
P\subseteq E$ and $e\dg$ is the smallest sharp element $p\in E$ such that $e\leq p$; hence, $E$ is a \emph{sharply
dominating} effect algebra and therefore it forms a so-called \emph{Brouwer-Zadeh poset} \cite{Gudsd}. (See
Section \ref{sc:K&BZ} below.)

By \cite[Theorem 9.1]{FPba} a \emph{generalized Hermitian algebra} (GH-algebra) \cite{FPgh} is the same thing as
a synaptic algebra $A$ that satisfies the following condition: Every bounded increasing sequence $a\sb{1}\leq
a\sb{2}\leq a\sb{3}\leq\cdots$ of pairwise commuting elements in $A$ has a supremum $a$ in $A$.

\section{Spectral order }

The assumption that $A$ is a synaptic algebra, $P$ is the OML  of projections in $A$, and $E$ is the algebra of effects
in $A$ remains in force.

Every  $a\in A$ defines and is determined by its \emph{spectral resolution} $(p_{a,\lambda})_{\lambda\in {\mathbb R}}$
in $P\cap CC(a)$, where $p_{a,\lambda}:=1-((a-\lambda)^+)\dg=(((a-\lambda)\sp{+})\dg)\sp{\perp}$ for $\lambda\in\reals$.
Also, $L_a:=\sup\{\lambda\in\reals:\lambda\leq a\}=\sup\{\lambda\in\reals:p_{a,\lambda}=0\}\in\reals$, $U_a:=\inf
\{\lambda\in\reals:a\leq\lambda\}=\inf\{\lambda\in\reals:p_{a,\lambda}=1\}\in\reals$, and
\[
a=\int_{L_a-0}^{U_a}\lambda dp\sb{a,\lambda},
\]
where the Riemann-Stieltjes type sums converge to the integral in norm. Two elements in $A$ commute iff their respective
spectral resolutions commute pairwise \cite[\S 8]{Fsa}.

\begin{theorem} {\rm\cite[Theorem 8.4]{Fsa}} \label{th:specresprops}
The spectral resolution $(p_{a,\lambda})_{\lambda\in\reals}$ of an element $a\in A$ has the following
properties for all $\lambda,\mu\in\reals${\rm: (i)} $p\sb{a,\lambda}(a-\lambda)\leq 0\leq(1-p\sb{a,\lambda})
(a-\lambda)$. {\rm(ii)} $\lambda \leq\mu\ \Rightarrow\ p_{a,\lambda}\leq p_{a,\mu}$. {\rm(iii)}
$\mu\geq U\sb{a}\ \Rightarrow\ p_{a,\mu}=1$. {\rm(iv)} $\mu <L\sb{a}\ \Rightarrow\ p_{a,\mu}=0$ and
$L\sb{a}<\mu\ \Rightarrow 0<p\sb{a,\mu}$. {\rm(v)} If $\alpha\in\reals$, then $p_{a,\alpha}=
\bigwedge_{\mu \in {\mathbb R}} \{ p_{a,\mu}:\alpha <\mu\}$.
\end{theorem}

\begin{lemma} \label{le:SpecResCarrier}
If $0\leq a\in A$, then $a\dg=1-p\sb{a,0}$.
\end{lemma}

\begin{proof}
Since $0\leq a$, we have $a=a\sp{+}$, whence $1-p\sb{a,0}=((a-0)\sp{+})\dg=a\dg$.
\end{proof}

In the following definition, we introduce the \emph{spectral order} $\leq\sb{s}$ on $A$ \cite{dG, Ols}.

\begin{definition}\label{de:specord}
For $a,b\in A$, $a\leq_s b$ iff for all $\lambda \in {\mathbb R}$, $p_{b,\lambda}\leq p_{a,\lambda}$.
\end{definition}

According to the next lemma, the spectral order agrees with the synaptic order on the OML $P$.

\begin{lemma}\label{le:proj}
For $q,r \in P$, $q\leq r \, \Leftrightarrow\, q\leq_sr$.
\end{lemma}

\begin{proof} The spectral resolution for $q\in P$ is
\[
p_{q,\lambda}=\left\{\begin{array}{lll}
                      0 & \,\mbox{if $\lambda < 0$}\\
                      1-q  & \mbox{ if $0\leq \lambda <1$}\\
                      1 &    \mbox{ if $1\leq \lambda$}
                       \end{array}
                       \right.
                       \]
Thus the desired result follows from the fact that $q\leq r$ iff $1-r\leq 1-q$.
\end{proof}

\begin{lemma} \label{le:alphaa+beta}
Let $\alpha, \beta\in\reals$ with $0<\alpha$ and let $a\in A$. Then, for all $\lambda
\in\reals${\rm: (i)} $p\sb{\alpha a,\lambda}=p\sb{a,\alpha\sp{-1}\lambda}$. {\rm (ii)}
$p\sb{a+\beta,\lambda}=p\sb{a,\lambda-\beta}$. {\rm(iii)} $p\sb{\alpha a+\beta,\lambda}
=p\sb{a,\alpha\sp{-1}(\lambda-\beta)}$.
\end{lemma}

\begin{proof}
To prove (i), we begin by noting that, for $0<\alpha$, $(\alpha a)\sp{+}=\frac12(|\alpha a|
+\alpha a)=\alpha(\frac12(|a|+a))=\alpha a\sp{+}$, from which it follows that, for all $\lambda
\in\reals$, $(\alpha a-\lambda)\sp{+}=(\alpha(a-\alpha\sp{-1}\lambda))\sp{+}=\alpha(a-\alpha\sp{-1}
\lambda)\sp{+}$. Also, for all $a\in A$, $(\alpha a)\dg=a\dg$, whence
\[
p\sb{\alpha a,\lambda}=(((\alpha a-\lambda)\sp{+})\dg)\sp{\perp}=((\alpha(a-\alpha\sp{-1}\lambda)\sp{+})
\dg)\sp{\perp}=(((a-\alpha\sp{-1}\lambda)\sp{+})\dg)\sp{\perp}=p\sb{a,\alpha\sp{-1}\lambda}.
\]
The proof of (ii) is obvious, and (iii) follows from (i) and (ii).
\end{proof}

\begin{lemma}\label{le:E} \rm{(\cite[Lemma 2.1]{dG})}
Let $\alpha, \beta\in\reals,\ \alpha>0$. Then for $a,b\in A${\rm:}
\[
{\rm(i)\ } a\leq b\Leftrightarrow\alpha a +\beta\leq\alpha  b+\beta.\ \  {\rm(ii)\ } a\leq_s b\Leftrightarrow
 \alpha a +\beta\leq_s\alpha b +\beta.
\]
\end{lemma}

\begin{proof} (i) is evident.  By Lemma \ref{le:alphaa+beta} (iii) $p\sb{\alpha a+\beta,\lambda}=
p\sb{a,\alpha\sp{-1}(\lambda-\beta)}$, and likewise $p\sb{\alpha b+\beta,\lambda}=p\sb{b,\alpha\sp{-1}
(\lambda-\beta)}$, from which (ii) follows immediately.
\end{proof}

\begin{remark} \label{rm:E}
In view of Lemma \ref{le:E}, for many purposes, it suffices to study the synaptic and spectral order on
elements of $E$. Indeed, suppose that $a\in A$. Then $-\|a\|\leq a\leq\|a\|$, and choosing any $n\in\Nat$ with
$\|a\|\leq n$, we have $-n\leq a\leq n$, which implies $\frac{1}{2n}a+\frac12\in E$. Thus, given $b\in A$,
we can choose $n\in\Nat$ with $\|a\|, \|b\|\leq n$, whereupon  $e:=\frac{1}{2n}a+\frac12\in E$, $f:=\frac
{1}{2n}b+\frac12\in E$ and $a\leq b$ (respectively, $a\leq\sb{s}b$) iff $e\leq f$ (respectively, iff $e\leq
\sb{s}f$).
\end{remark}

\begin{lemma} \label{le:esubn}
Let $e\in E$ and for each $n\in\Nat$, let $e\sb{n}:=\sum\sb{k=1}\sp{n}\frac{k}{n}(p\sb{e,\frac{k}{n}}-
p\sb{e,\frac{k-1}{n}})$. Then {\rm(i)} $\lim\sb{n\rightarrow\infty}e\sb{n}=e$ and {\rm(ii)} $e\sb{n}=
1-\frac1n\sum\sb{k=0}\sp{n-1}p\sb{e,\frac{k}{n}}$.
\end{lemma}

\begin{proof}
For each $n\in\Nat$ we form the partition $\lambda\sb{0}<\lambda\sb{1}<\cdots<\lambda\sb{n}
<\lambda\sb{n+1}$ of the closed interval $[-\frac{1}{n},1]$ given by $\lambda\sb{i}:=(i-1)/n$
for $i=0,1,2,...,n,n+1$. Since $e\in E$, we have $\lambda\sb{0}<0\leq L\sb{e}\leq U\sb{e}
\leq 1=\lambda\sb{n+1}$. Put $\gamma\sb{1}:=0$ and for $1<i\leq n+1$ put $\gamma\sb{i}:=
\lambda\sb{i}=(i-1)/n$. Then $\lambda\sb{i-1}\leq\gamma\sb{i}\leq\lambda\sb{i}$ for $i=
1,2,...,n,n+1$, and we have
\[
\sum\sb{i=1}\sp{n+1}\gamma\sb{i}\left(p\sb{e,\lambda\sb{i}}-p\sb{e,\lambda\sb{i-1}}\right)
=\sum\sb{i=2}\sp{n+1}\gamma\sb{i}\left(p\sb{e,\lambda\sb{i}}-p\sb{e,\lambda\sb{i-1}}\right)=
\]
\[
\sum\sb{k=1}\sp{n}\lambda\sb{k+1}\left(p\sb{e,\lambda\sb{k+1}}-p\sb{e,\lambda\sb{k}}\right)
=\sum\sb{k=1}\sp{n}\frac{k}{n}\left(p\sb{e,\frac{k}{n}}-p\sb{e,\frac{k-1}{n}}\right)=e\sb{n}.
\]
Therefore, by \cite[Theorem 3.1]{FPSROUS} and \cite[Remark 3.1]{FPSROUS} (which takes care of
the case $U\sb{e}=1$), we have $\lim\sb{n\rightarrow\infty}e\sb{n}=e$, proving (i).

Also,
\[
e\sb{n}:=\sum\sb{k=1}\sp{n}\frac{k}{n}\left(p\sb{e,\frac{k}{n}}-p\sb{e,\frac{k-1}{n}}\right)
=\sum\sb{k=1}\sp{n}\frac{k}{n}p\sb{e,\frac{k}{n}}-\sum\sb{k=0}\sp{n-1}\frac{k+1}{n}p\sb{e,
\frac{k}{n}}=
\]
\[
p\sb{e,1}+\sum\sb{k=1}\sp{n-1}\frac{k}{n}p\sb{e,\frac{k}{n}}-\sum\sb{k=1}\sp{n-1}\frac{k+1}{n}p\sb{e,
\frac{k}{n}}-\frac1n p\sb{e,0}
\]
\[
=p\sb{e,1}-\sum\sb{k=1}\sp{n-1}\frac{1}{n}p\sb{e,\frac{k}{n}}-\frac{1}{n}p\sb{e,0}=1-\frac1n\sum
\sb{k=0}\sp{n-1}p\sb{e,\frac{k}{n}}.
\]
proving (ii).
\end{proof}

\begin{theorem}\label{th:le-les} For all $a ,b\in A$, $a\leq_s b\, \Rightarrow \, a\leq b$.
\end{theorem}

\begin{proof}
Let $e,f\in E$ with $e\leq\sb{s}f$. By Remark \ref{rm:E}, it will be enough to prove that
$e\leq f$. For each $n\in\Nat$, put $e\sb{n}:=1-\frac1n\sum\sb{k=0}\sp{n-1}p\sb{e,\frac{k}{n}}$
and $f\sb{n}:=1-\frac1n\sum\sb{k=0}\sp{n-1}p\sb{f,\frac{k}{n}}$. Then, since $e\leq\sb{s}f$,
we have $e\sb{n}\leq f\sb{n}$, i.e., $0\leq f\sb{n}-e\sb{n}$ for all $n\in\Nat$. By Lemma
\ref{le:esubn}, $\lim\sb{n\rightarrow\infty}(f\sb{n}-e\sb{n})=f-e$, and by \cite[Theorem
4.7 (iii)]{Fsa}, $0\leq f-e$, i.e., $e\leq f$.
\end{proof}

\begin{lemma} \label{le:inequality}
If $a,b\in A$, $aCb$, and $a\leq b$, then $a\sp{+}\leq b\sp{+}$.
\end{lemma}

\begin{proof}
Assume the hypotheses of the lemma. Then, since $0\leq b\sp{-}$, we have $a\leq b=b\sp{+}-b\sp{-}
\leq b\sp{+}$. Also, since $b\sp{+}\in CC(b)$ and $bCa$, we have $b\sp{+}Ca$. Put $p:=(a\sp{+})\dg$.
Then by parts (i) and (iii) of \cite[Theorem 3.3]{Fsa}, $a\sp{+}=pa$ and $p\in CC(a)$, whence $pCa$,
$pCb\sp{+}$, and $(1-p)Cb\sp{+}$. Now, $0\leq p$, $0\leq b\sp{+}-a$, and $pC(b\sp{+}-a)$, so by \cite
[Lemma 1.5]{Fsa}. $0\leq p(b\sp{+}-a)=pb\sp{+}-pa$, i.e., $a\sp{+}=pa\leq pb\sp{+}$. Likewise, by
\cite[Lemma 1.5]{Fsa} $0\leq(1-p)b\sp{+}=b\sp{+}-pb\sp{+}$, i.e., $pb\sp{+}\leq b\sp{+}$, and it
follows that $a\sp{+}\leq b\sp{+}$.
\end{proof}

\begin{theorem}\label{th:commut} If $a,b\in A$ and $aCb$, then $a\leq b\Leftrightarrow a\leq_s b$.
\end{theorem}

\begin{proof}
Assume that $a,b\in A$ and $aCb$. By Theorem \ref{th:le-les}, we only need to prove that
$a\leq b\Rightarrow a\leq_s b$, so assume that $a\leq b$ and $\lambda\in\reals$. Then
$a-\lambda\leq b-\lambda$, whence $0\leq (a-\lambda)\sp{+}\leq(b-\lambda)\sp{+}$ by Lemma
\ref{le:inequality}. Therefore, by \cite[Theorem 2.10 (viii)]{Fsa}, $((a-\lambda)\sp{+})\dg
\leq((b-\lambda)\sp{+})\dg$, whence $p\sb{b,\lambda}=1-((b-\lambda)\sp{+})\dg\leq 1-((a-
\lambda)\sp{+})\dg=p\sb{a,\lambda}$, and it follows that $a\leq\sb{s}b$.
\end{proof}

\begin{corollary}\label{co:oneproj} Let $e,f\in E$ and $e\in P$ or $f\in P$. Then $e\leq f\,
\Leftrightarrow\, e\leq\sb{s}f$.
\end{corollary}

\begin{proof} Assume that $e\in P$, $f\in E$. Then $e\leq f$ iff $ef=fe=e$, hence $eCf$, and
apply Theorem \ref{th:commut}. If $e\in E, f\in P$, then $e\leq f$ iff $ef=fe=f$, and again
$eCf$.
\end{proof}

\begin{corollary} \label{le:Especorder}
$E=\{e\in A:0\leq\sb{s}e\leq\sb{s}1\}$.
\end{corollary}

\begin{proof}
If $e\in A$, then both $0$ and $1$ commute with $e$.
\end{proof}

\section{Lattice properties of the spectral order on a Banach synaptic algebra}

\emph{In this section, $A$ denotes a Banach SA.}

As we mentioned above, an SA is Banach (i.e., norm complete) iff it is isomorphic to the
self-adjoint part of a Rickart C$\sp{\ast}$-algebra \cite[Theorem 5.3]{FPba}; moreover the
set of projections in a Banach SA is a $\sigma$-complete OML \cite[Theorem 5.4]{FPba}.
Furthermore, if an SA is Banach, then it is isomorphic to the self-adjoint part of an
AW$\sp{\ast}$-algebra iff its set of projections is a complete OML \cite[Theorem 8.5]{FPba}.

The definitions and results in this section are versions for a Banach SA of corresponding
definitions and results for von Neumann algebras in \cite[\S 5]{KR} and \cite{dG}.

\begin{definition} {\cite[p. 311]{KR}} \label{de:resid} A family of projections $(p_{\lambda})_{\lambda
\in\reals}$ is a \emph{bounded resolution of identity} iff the following hold for $\lambda, \lambda
\sp{\prime}\in\reals$:
\begin{enumerate}
\item[(1)] There exists $0\leq K\in\reals$ such that $p_{\lambda}=0$ if $\lambda<-K$ and $p_{\lambda }=1$
if $K\leq\lambda$.
\item[(2)] $p_{\lambda}\leq p_{\lambda'}$ if $\lambda \leq \lambda'$.
\item[(3)] $p_{\lambda}=\bigwedge_{\lambda'>\lambda }p_{\lambda'}$.
\end{enumerate}
\end{definition}
Note that condition (2) in Definition \ref{de:resid} is redundant---it follows from condition (3).
Also note that, by condition (2), the projections in a bounded resolution of the identity are
mutually commutative.

With only slight and obvious modifications, the proofs of \cite[Theorems 5.2.3 and 5.2.4]{KR} work
in the Banach SA $A$, and we have the following.

\begin{theorem}\label{th:resid} If $(p_{\lambda})_{\lambda \in {\mathbb R}}$ is a bounded resolution of
identity in a Banach SA, then the Riemann-Stieltjes sums for $\int_{-K}^K \lambda dp_{\lambda}$ converge
in norm to an element $a=\int_{-K}^K \lambda dp_{\lambda}\in A$ such that $\|a\|\leq K$ and for which
$(p_{\lambda})_{\lambda\in\reals}$ is the spectral resolution of $a$.
\end{theorem}

In part (ii) of the following lemma, the infimum $\bigwedge_{\lambda<\mu\in\reals}(p_{\mu}\vee q_{\mu})$
coincides with the countable infimum $\bigwedge_{\lambda<\xi\in\ratnls}(p_{\xi}\vee q_{\xi})$, $\ratnls=$
the rational numbers, which exists since $P$ is $\sigma$-complete. A similar remark applies to other analogous
infima that appear below.

\begin{lemma}\label{le:supinf} Let $(p_{\lambda})_{\lambda\in\reals}$ and $(q_{\lambda})_{\lambda\in
\reals}$ be bounded resolutions of identity in the Banach SA $A$.
Then,
\begin{enumerate}
\item $u_{\lambda}:=p_{\lambda}\wedge q_{\lambda}$, $\lambda\in\reals$, and
\item $v_{\lambda}:=\bigwedge_{\lambda<\mu\in\reals}(p_{\mu}\vee q_{\mu})$, $\lambda\in\reals$.
\end{enumerate}
\noindent define bounded resolutions of identity $(u_{\lambda})_{\lambda \in\reals}$ and
$(v_{\lambda})_{\lambda\in\reals}$.
\end{lemma}

\begin{proof} Parts (1) and (2) of Definition \ref{de:resid} are clear as is part (3) for (i). We only have to
check part (3) for (ii). We have
\[
\bigwedge\sb{\lambda\sp{\prime}>\lambda}v\sb{\lambda\sp{\prime}}=\bigwedge_{\lambda'>\lambda}
\left(\bigwedge_{\mu >\lambda'}(p_{\mu}\vee q_{\mu})\right)=\bigwedge_{\mu > \lambda}(p_{\mu}\vee q_{\mu})
=v_{\lambda}. \qedhere
\]
\end{proof}

\begin{definition}\label{de:supinf} Let $a,b$ be elements of the Banach SA $A$ with corresponding spectral
resolutions $(p_{a,\lambda})_{\lambda\in\reals}$, $(p_{b,\lambda})_{\lambda\in\reals}$. Then, as per Lemma
\ref{le:supinf}, we define $a\vee_s b$ and $a\wedge_s b$ to be the elements of $A$ whose spectral resolutions
are given by $p\sb{a\vee_sb,\lambda}:=p_{\lambda}\wedge q_{\lambda}$, $\lambda\in\reals$ and $p\sb{a\wedge_s b,
\lambda}:=\bigwedge_{\lambda<\mu\in\reals}(p_{\mu}\vee q_{\mu})$, $\lambda\in\reals$.
\end{definition}

\begin{theorem} \label{th:minmax} For $a,b$ in the Banach SA $A$, the element $a\wedge_sb$ is the infimum
and $a\vee_s b$ is the supremum of $a$ and $b$ in the spectral order $\leq_s$, whence $(A;\leq\sb{s})$ is
a lattice. Moreover, under the spectral order, $E$ is a sublattice of $A$.
\end{theorem}

\begin{proof} For all $\lambda \in {\mathbb R}$,
\[
p_{a\wedge_s b,\lambda}=\bigwedge_{\mu > \lambda}(p_{a,\mu}\vee p_{b,\mu})\geq \bigwedge_{\mu>\lambda}p_{a,\mu},
\]
which shows that $a\wedge_sb\leq_sa$, and similarly $a\wedge_s b\leq_s b$. If $c\in A$ is such that $c\leq_s a,b$,
then $p_{a,\mu}, p_{b,\mu}\leq p_{c,\mu}$ for all $\mu$, hence
\[
\bigwedge_{\mu>\lambda}(p_{a,\mu}\vee p_{b,\mu})\leq \bigwedge_{\mu>\lambda}p_{c,\mu}=p_{c,\lambda}.
\]
This shows that $c\leq_s a\wedge_sb$. A similar argument proves that $a,b\leq_s a\vee_s b$ and that $a,b
\leq_s c$ implies $a\vee_s b\leq c$.

By Corollary \ref{le:Especorder}, under the spectral order, $E$ is a sublattice of $A$.
\end{proof}

\begin{theorem}\label{th:Sigmalat}
For the Banach SA $A$, $(E;\leq_s)$ is a  $\sigma$-complete lattice.
\end{theorem}

\begin{proof} Let $(e_n)_{n\in\Nat}$ be a countable family in $E$ and let $(p_{e_n,\lambda})_{\lambda
\in\reals}$ be the corresponding spectral resolutions of $e_n, n\in\Nat$.

Let
\[
u_{\lambda}:=\bigwedge_{n\in\Nat}p_{e_n,\lambda}.
\]
It is easy to check that $(u_{\lambda})_{\lambda \in\reals}$ is a spectral resolution of an element
$e^{\vee}$ belonging to $E$. From the definition of $u_{\lambda}$ we have  $p_{e_n,\lambda}\geq u_{\lambda},\,
\forall\lambda \, \Rightarrow \, e_n\leq_s e^{\vee}$ for all $n\in\Nat$. Let $f\in E$ with the spectral
resolution $(p_{f,\lambda})_{\lambda \in\reals}$ be such that $e_n\leq_s f,\, \forall n\in\Nat$, i.e.,
\[
\forall n\in\Nat, \forall \lambda\in\reals: p_{f,\lambda}\leq p_{e_n,\lambda}.
\]
Then
\[
\forall\lambda\in\reals, \  p_{f,\lambda}\leq \bigwedge_{n\in\Nat} p_{e_n,\lambda}=u_{\lambda},
 \]
i.e., $e^{\vee}\leq_s f$. Therefore $\overset{s}{\bigvee}_{n\in\Nat}e_n:=e^{\vee}$ is the supremum of
$(e_n)_{n\in\Nat}$ in $(E;\leq_s)$.

To show that $(e_n)_{n\in\Nat}$ has an infimum in $E$, we set
 \[
 v_{\lambda}:=\bigwedge_{\mu>\lambda}\bigvee_{n\in\Nat} p_{e_n,\mu}.
 \]

We shall show that $(v_{\lambda})_{\lambda\in\reals}$ is a spectral resolution. The properties $v_{\lambda}=0$
for $\lambda < 0$ and $v_{\lambda} =1$ for $\lambda > 1$ are clear. Let $\mu, \nu, \lambda_1,\lambda_2$ be such
that $\lambda_1< \mu< \lambda_2<\nu$. Then
 \[
 \bigvee_{n\in\Nat} p_{e_n,\mu}\leq \bigvee_{n\in\Nat}p_{e_n,\nu}, \text{\ whence\ }\bigvee_{n\in\Nat}p_{e_n,\mu}\leq \bigwedge_
 {\nu>\lambda_2}\bigvee_{n\in\Nat}p_{e_n,\nu}.
 \]
This implies
 \[
\bigwedge_{\mu>\lambda_1}\bigvee_{n\in\Nat}p_{e_n,\mu}\leq\bigwedge_{\nu>\lambda_2}\bigvee_{n\in\Nat}
p_{e_n,\nu}, \text{\ i.e.\ }v_{\lambda_1}\leq v_{\lambda_2}.
 \]

We also have
\[
 \bigwedge_{\mu>\lambda} v_{\mu}=\bigwedge_{\mu>\lambda}\bigwedge_{\nu>\mu}\bigvee_{n\in\Nat}p_{e_n,\nu}=
\bigwedge_{\nu>\lambda}\bigvee_{n\in\Nat}p_{e_n,\nu}=v_{\lambda}.
 \]
Hence $(v_{\lambda})_{\lambda\in\reals}$ is the spectral resolution of an element $e^{\wedge}\in E$.
It remains to prove that $e^{\wedge}$ is the infimum of $(e_n)_{n\in\Nat}$. Let $f\in E$ with the spectral
resolution $(q_{\lambda})_{\lambda\in\reals}$ satisfy $f\leq_s e_n,\, \forall n\in\Nat$. Then $p_{e_n,\mu}
\leq q_{\mu},\, \forall \mu$, hence
\[
 p_{e^{\wedge},\lambda}=v_{\lambda}=\bigwedge_{\mu>\lambda}\bigvee_n p_{e_n,\mu} \leq \bigwedge_{\mu>\lambda}
q_{\mu}=q_{\lambda},\, \forall  \lambda,
\]
hence $e^{\wedge}\geq_s f$. This proves that $e^{\wedge}$ is the infimum of $(e_n)_{n\in\Nat}$.
\end{proof}

In view of Remark \ref{rm:E}, Theorem \ref{th:Sigmalat} has the following corollary.

\begin{corollary} \label{co:Sigmalat}
For the Banach SA $A$, $(A;\leq_s)$ is a  Dedekind $\sigma$-complete lattice.
\end{corollary}

\begin{remarks}
Arguments similar to the proofs of Theorem \ref{th:Sigmalat} and Corollary \ref{co:Sigmalat} show that, if
$A$ is isomorphic to the self-adjoint part of an AW$\sp{\ast}$-algebra (i.e., $A$ is Banach and $P$
is a complete OML), then under the spectral order, the algebra $E$ of effects forms a complete lattice and $A$
is a Dedekind complete lattice (Cf. \cite[Theorem 3.1]{dG}).
\end{remarks}

\section{Kleene and Brouwer-Zadeh posets} \label{sc:K&BZ}

\begin{definition} \label{df:involution}
If $({\mathcal P};\leq)$ is a poset, then an \emph{involution} on ${\mathcal P}$ is a mapping
$\sp{\perp}\colon{\mathcal P}\to{\mathcal P}$ such that, for all $a,b\in{\mathcal P}$:
\begin{enumerate}
\item[(I1)] $(a\sp{\perp})\sp{\perp}=a$ (i.e., $\sp{\perp}$ is of \emph{period two}), and
\item[(I2)] $a\leq b\Rightarrow b\sp{\perp}\leq a\sp{\perp}$ (i.e., $\sp{\perp}$ is \emph{order-reversing}).
\end{enumerate}
An involution $\sp{\perp}\colon{\mathcal P}\to{\mathcal P}$ is \emph{regular} iff, for all $a,b\in
{\mathcal P}$,
\begin{enumerate}
\item[(R)] If $a\leq a\sp{\perp}$ and $b\leq b\sp{\perp}$, then $a\leq b\sp{\perp}$.
\end{enumerate}
An \emph{involution poset} is a structure $({\mathcal P};\leq, \sp{\perp})$ where $\sp{\perp}$
is an involution on the poset $({\mathcal P};\leq)$. If $\sp{\perp}$ is regular, then $({\mathcal P};
\leq,\sp{\perp})$ is a \emph{regular involution poset}. An \emph{involution lattice} is an involution poset
$({\mathcal P};\leq,\sp{\perp})$ such that $({\mathcal P};\leq)$ is a lattice. A \emph{bounded involution
poset} is a structure $({\mathcal P};\leq, ^{\perp},0,1)$, where $(L,\leq, 0,1)$ is a bounded poset (i.e.,
$0\leq a\leq 1,\,\forall a\in{\mathcal P}$) and $({\mathcal P};\leq, ^{\perp})$ is an involution poset.
A bounded involution poset $({\mathcal P};\leq, ^{\perp},0,1)$ such that the involution $^{\perp}$ is
regular is called a \emph{Kleene poset}. A Kleene poset $({\mathcal P};\leq,^{\perp},0,1)$ is a \emph
{Kleene lattice} iff $({\mathcal P};\leq)$ is a lattice.
\end{definition}

If $({\mathcal P};\leq, \sp{\perp})$ is an involution poset, then there is a \emph{De Morgan duality}
on ${\mathcal P}$ whereby for each existing infimum $\bigwedge\sb{\gamma\in\Gamma}a\sb{\gamma}$ in
${\mathcal P}$, the supremum $\bigvee\sb{\gamma\in\Gamma}a\sb{\gamma}\sp{\perp}$ exists in
${\mathcal P}$ and $\left(\bigwedge\sb{\gamma\in\Gamma}a\sb{\gamma}\right)\sp{\perp}=\bigvee\sb{\gamma
\in\Gamma}a\sb{\gamma}\sp{\perp}$; likewise, for each existing supremum $\bigvee\sb{\gamma\in\Gamma}a
\sb{\gamma}$ in ${\mathcal P}$, the infimum $\bigwedge\sb{\gamma\in\Gamma}a\sb{\gamma}\sp{\perp}$ exists
in ${\mathcal P}$ and $\left(\bigvee\sb{\gamma\in\Gamma}a\sb{\gamma}\right)\sp{\perp}=\bigwedge
\sb{\gamma\in\Gamma}a\sb{\gamma}\sp{\perp}$. In particular, an involution poset $({\mathcal P};
\leq,\sp{\perp})$ satisfies the \emph{De\,Morgan laws}: For all $a,b\in{\mathcal P}$, if $a\wedge b$
exists in ${\mathcal P}$, then $a\sp{\perp}\vee b\sp{\perp}$ exists in ${\mathcal P}$ and $(a\wedge b)
\sp{\perp}=a\sp{\perp}\vee b\sp{\perp}$; likewise, if $a\vee b$ exists in ${\mathcal P}$, then
$a\sp{\perp}\wedge b\sp{\perp}$ exists in ${\mathcal P}$ and $(a\vee b)\sp{\perp}=a\sp{\perp}\wedge b
\sp{\perp}$.

The assumption that $A$ is a synaptic algebra, $E$ is the algebra of effects in $A$, and $P$ is the OML
of projections in $A$ remains in force. We shall need the next lemma for the proof of part (ii) of Theorem
\ref{th:EKleene} below. In \cite[Definition 8.2 (iii)]{Fsa}, for $a\in A$ and $\lambda\in\reals$, $d
\sb{a,\lambda} :=1-(a-\lambda)\dg=((a-\lambda)\dg)\sp{\perp}$, which plays a role in the statement of the
lemma and its proof, is called the \emph{$\lambda$-eigenprojection} of $a$, and $\lambda$ is an \emph
{eigenvalue} of $a$ iff $d\sb{a,\lambda}\not=0$.

\begin{lemma} \label{le:SRes-a}
Let $a\in A$ and $\lambda\in\reals$. Then{\rm: (i)} $p\sb{-a,\lambda}=1-p\sb{a,-\lambda}+
d\sb{a,-\lambda}$. {\rm(ii)} $p\sb{-a,\lambda}=1-\bigvee\sb{\mu\in\reals,\, \mu<-\lambda}p\sb{a,\mu}$.
{\rm(iii)} $p\sb{1-a,\lambda}=1-\bigvee\sb{\mu\in\reals,\,\mu<1-\lambda}p\sb{a,\mu}$.
\end{lemma}

\begin{proof}
Evidently,
\[
1-p\sb{-a,\lambda}=((-a-\lambda)\sp{+})\dg=((-(a-(-\lambda))\sp{+})\dg=(a-(-\lambda))\sp{-})\dg
=((a+\lambda)\sp{-})\dg.
\]
Therefore, by \cite[Theorem 3.3 (viii)]{Fsa},
\[
(1-p\sb{a,-\lambda})+(1-p\sb{-a,\lambda})=((a+\lambda)\sp{+})\dg+((a+\lambda)\sp{-})\dg=
(a+\lambda)\dg=1-d\sb{a,-\lambda},
\]
from which (i) follows. As a consequence of \cite[Theorem 8.4 (viii)]{Fsa}, we have
$d\sb{a,-\lambda}=p\sb{a,-\lambda}-\bigvee\sb{\mu\in\reals,\,\mu<-\lambda}p\sb{a,\mu}$,
which, together with (i), yields (ii). Part (iii) follows immediately from (ii) and
Lemma \ref{le:alphaa+beta} (ii).
\end{proof}

\begin{theorem} \label{th:EKleene}
We extend the orthosupplementation $\sp{\perp}\colon E\to E$ to all of the SA $A$ by defining
$a\sp{\perp}:=1-a$ for all $a\in A$. Then{\rm:}

\begin{enumerate}
\item $(A;\leq,\sp{\perp})$ is a regular involution poset and $(E;\leq,\sp{\perp},0,1)$ is a Kleene
poset.
\item $(A;\leq\sb{s},\sp{\perp})$ is a regular involution poset and $(E;\leq\sb{s},\sp{\perp},0,1)$
is a Kleene poset. Moreover, if $A$ is Banach, then $(A;\leq\sb{s},\sp{\perp})$ is a regular involution
lattice and $(E;\leq\sb{s},\sp{\perp},0,1)$ is a Kleene lattice.
\end{enumerate}
\end{theorem}

\begin{proof}
(i) For the synaptic order $\leq$, it is clear that the mapping $\sp{\perp}\colon A\to A$ satisfies
conditions (I1) and (I2) of Definition \ref{df:involution}, whence it is an involution on $A$. Moreover,
if $a,b\in A$, with $a\leq(1-a)$, and $b\leq(1-b)$, then $a\leq\frac12$ and $\frac12\leq(1-b)$, hence $a
\leq 1-b$, so (R) is satisfied. Clearly, then, $(E;\leq,\sp{\perp},0,1)$ is a Kleene poset.

(ii) For the spectral order, condition (I1) in Definition \ref{df:involution} is obvious. To prove
condition (I2), suppose that $a,b\in A$ and $a\leq\sb{s}b$. Then $p\sb{b,\mu}\leq p\sb{a,\mu}$ for all
$\mu\in\reals$, whence, for all $\lambda\in\reals$, $1-\bigvee\sb{\mu\in\reals,\,\mu<1-\lambda}p
\sb{a,\mu}\leq1-\bigvee\sb{\mu\in\reals,\,\mu<1-\lambda}p\sb{b,\mu}$, and it follows from Lemma
\ref{le:SRes-a} (iii) that $p\sb{1-a,\lambda}\leq p\sb{1-b,\lambda}$, i.e., $b\sp{\perp}
\leq\sb{s}a\sp{\perp}$. To prove condition (R), suppose that $a,b\in A$ with $a\leq
\sb{s}(1-a)$ and $b\leq\sb{s}(1-b)$.  Since $aC(1-a)$ and $aC\frac12$, we have $a\leq_s 1-a$
iff $a\leq 1-a$ iff $a\leq\frac12$ iff $a\leq_s\frac12$. Likewise, since $bC(1-b)$ and
$\frac12 C(1-b)$, we have $b\leq_s(1-b)$ iff $b\leq(1-b)$ iff $\frac12\leq(1-b)$ iff
$\frac12\leq_s(1-b)$. Therefore, if $a\leq\sb{s}(1-a)$ and $b\leq\sb{s}(1-b)$, then $a\leq_s
\frac12$ and $\frac12\leq_s(1-b)$, whence $a\leq_s(1-b)$, proving (R). Therefore, $(A;\leq\sb{s},
\sp{\perp})$ is a regular involution poset and $(E;\leq\sb{s},\sp{\perp},0,1)$ is a Kleene poset.

If $A$ is Banach, then by Theorem \ref{th:minmax}, $A$ is a lattice under $\leq\sb{s}$, whence
it is a regular involution lattice and $(E;\leq\sb{s},\sp{\perp},0,1)$ is a Kleene lattice.
\end{proof}

\begin{definition}\label{de:bzp} {\rm \cite[Definition 4.1]{GLP}}
A \emph{Brouwer-Zadeh poset} (BZ-poset for short) is a structure $({\mathcal P};\leq, ^{\perp},
^{\sim}, 0,1)$ where
\begin{enumerate}
\item[(1)] $({\mathcal P};\leq, \sp{\perp}, 0, 1)$ is a Kleene poset.
\item[(2)] $\sp{\sim}$ is a unary operation on ${\mathcal P}$, called the \emph{Brouwer
complementation}, such that, for all $a,b\in{\mathcal P}$:
\begin{enumerate}
 \item[(2a)] $a\wedge a^{\sim}=0$.
 \item[(2b)] $a\leq a^{\sim \sim}$.
 \item[(2c)] $a\leq b \ \Rightarrow\ b^{\sim} \leq a^{\sim}$.
\end{enumerate}
\item[(3)] $a^{\sim \perp} = a^{\sim \sim}$.
\end{enumerate}
A BZ-poset that is a lattice is called a \emph{BZ-lattice}.
\end{definition}

\begin{lemma}\label{le:e^o} Let $a,b\in A$. Then $0\leq a\Leftrightarrow0\leq\sb{s}a$ and
\[
\text{if\ }0\leq a,\text{\ then\ }a\leq\sb{s}b\Rightarrow a\leq b\Rightarrow a\dg\leq b\dg
\Leftrightarrow a\dg\leq\sb{s}b\dg.
\]
\end{lemma}

\begin{proof}
Since $0Ca$, we have $0\leq a\Leftrightarrow0\leq\sb{s}a$. Assume that $0\leq a$. Then
$a\leq\sb{s}b\Rightarrow0\leq a\leq b$ and by \cite[Theorem 2.12 (viii)]{Fsa}, $0\leq a\leq b
\Rightarrow a\dg\leq b\dg$. Since both $a\dg$ and $b\dg$ are projections, it follows that
$a\dg\leq b\dg\Leftrightarrow a\dg\leq\sb{s}b\dg$.
\end{proof}

\begin{theorem} \label{th:BZ}
For $e\in E$, define $e^{\sim}:=(e\dg)\sp{\perp}$. Then{\rm:}
\begin{enumerate}
\item $e\in E\Rightarrow e^{\sim \sim}=e\dg$ and $p\in P\Rightarrow p^{\sim}=p\sp{\perp}$.
\item Both $(E;\leq,^{\perp},^{\sim},0,1)$ and $(E;\leq\sb{s},^{\perp},^{\sim},0,1)$ are
 Brouwer-Zadeh posets.
\item If $A$ is Banach, then $(E;\leq\sb{s},^{\perp},^{\sim},0,1)$ is a Brouwer-Zadeh lattice.
\end{enumerate}
\end{theorem}

\begin{proof}
(i) If $p\in P$, then $p\dg=p$. Therefore, as $(e\dg)\sp{\perp}$ is a projection, we have
$e^{\sim \sim}=(((e\dg)\sp{\perp})\dg)\sp{\perp}=((e\dg)\sp{\perp})\sp{\perp}=e\dg$. Similarly,
for $p\in P$, $p^{\sim}=(p\dg)\sp{\perp}=p\sp{\perp}$.

(ii) By Theorem \ref{th:EKleene}, we have part (1) of Definition \ref{de:bzp} for both
$\leq$ and $\leq\sb{s}$. Suppose that $e,f\in E$. To prove (2a), suppose $f\leq\sb{s}e,e
\sp{\sim}$. Then $f\leq e$, $f\leq(e\dg)\sp{\perp}$, and it will be sufficient to prove
that $f=0$. Since $f\leq(e\dg)\sp{\perp}$, it follows that $fe=ef=0$, whence $f\leq e$
implies that $0\leq f\sp{2}\leq fe=0$, whereupon $f=0$. By (i), to prove (2b), we have to
prove that, $e\leq e\dg$ and $e\leq\sb{s}e\dg$. But we have $e\leq e\dg$, and $eCe\dg$,
whence $e\leq\sb{s}e\dg$. To prove (2c), we begin by noting that by Theorem \ref{th:EKleene},
$\sp{\perp}\colon E\to E$ is order-reversing for both $\leq$ and $\leq\sb{s}$; moreover by
Lemma \ref{le:e^o}, $e\leq f$ implies that both $e\dg\leq f\dg$ and $e\dg\leq\sb{s}f\dg$
hold, from which (2c) follows. Finally, (3) is a simple consequence of (i).

(iii) Part (iii) follows from (ii) and the fact that, by Theorem \ref{th:EKleene} (ii), if
$A$ is Banach, then $(E;\leq\sb{s})$ is a lattice.
\end{proof}

\section{De\,Morgan laws}

By Theorem \ref{th:EKleene}, the Kleene complement $e\mapsto e\sp{\perp}=1-e$ is an involution
on both $(E;\leq,\sp{\perp},0,1)$ and $(E;\leq\sb{s},\sp{\perp},0,1)$, hence it satisfies De\,Morgan
laws for both orders $\leq$ and $\leq\sb{s}$. Furthermore, if $A$ is Banach, then by Theorem
\ref{th:EKleene} (ii), $(E;\leq\sb{s},\sp{\perp},0,1)$ is a Kleene lattice, whence for $e,f\in E$,
we have $(e\wedge_sf)\sp{\perp}=e\sp{\perp}\vee_s f\sp{\perp}$ and $(e\vee_s f)\sp{\perp}=e\sp{\perp}
\wedge_s f\sp{\perp}$.

By Theorem \ref{th:BZ}, both $(E;\leq,^{\perp},^{\sim},0,1)$ and $(E;\leq\sb{s},^{\perp},^{\sim}
,0,1)$ are BZ-posets, which raises the question of whether the Brouwer complement $e\mapsto e
\sp{\sim}=(e\dg)\sp{\perp}$ on $E$ satisfies De\,Morgan laws for $\leq$ and $\leq\sb{s}$. For
the synaptic order, the condition $(e\vee f)^{\sim}=e^{\sim}\wedge f^{\sim}$ would mean that
$(e\vee f)^{o \perp}=e^{o \perp}\wedge f^{o \perp}$, i.e. that $(e\vee f)^o=e^o\vee f^o$ if $e\vee f$ exists in $E$. Similarly,
the condition $(e\wedge f)^{\sim}=e^{\sim}\vee f^{\sim}$ would mean that $(e\wedge f)\dg=
e\dg\wedge f\dg$ if $e\wedge f$ exists in $E$. For the spectral order, analogous conditions $(e\vee\sb{s}f)\dg=e\dg\vee\sb{s}
f\dg$ and $(e\wedge\sb{s}f)\dg=e\dg\wedge\sb{s}f\dg$ would be equivalent to the De\,Morgan
laws for existing suprema $e\vee\sb{s}f$ and infima $e\wedge\sb{s}f$ in $(E;\leq\sb{s},\sp{\perp},0,1)$.
(Recall that $e\dg\vee\sb{s}f\dg=e\dg\vee f\dg$ and $e\dg\wedge\sb{s}f\dg=e\dg\wedge f\dg$.) In this
connection, we have the following result.

\begin{lemma}\label{le:dMvee}
Let $e,f\in E$. Then{\rm: (i)} If $e\vee f$ {\rm(}resp. $e\vee\sb{s}f${\rm)} exists in $E$, then
$(e\vee f)\dg=e\dg\vee f\dg$ {\rm(}resp. $(e\vee\sb{s}f)\dg=e\dg\vee\sb{s}f\dg${\rm)}.
{\rm (ii)} If $e\wedge f$ {\rm(}resp. $e\wedge\sb{s}f${\rm)} exists in $E$, then $(e\wedge f)
\dg\leq e\dg\wedge f\dg$ {\rm(}resp. $(e\wedge\sb{s}f)\dg\leq\sb{s}e\dg\wedge\sb{s}f\dg${\rm)}.
\end{lemma}

\begin{proof}
(i) From $0\leq e,f\leq e\vee f$ and Lemma \ref{le:e^o}, we obtain $e\dg\vee f\dg\leq(e\vee f)\dg$.
On the other hand, $0\leq e\vee f\leq e\dg\vee f\dg$ yields $(e\vee f)\dg\leq e\dg\vee f\dg$. A
similar argument applies to $e\vee\sb{s}f$.

(ii) From $e\wedge f\leq e,f$ we get $(e\wedge f)\dg\leq e\dg\wedge f\dg$ and similarly for
$e\wedge\sb{s}f$.
\end{proof}

In von Neumann algebras, the condition $(e\wedge f)\dg=e\dg\wedge f\dg$ holds with respect to the
numerical order for existing infima of effects $e$ and $f$ iff the von Neumann algebra is of finite
type \cite{CaHa}. A similar result for the spectral order is proved in \cite[Corollary 4.4]{dG}.
If $A$ is of finite type and the OML $P$ is complete, we prove in Theorem \ref{le:caha} that this condition holds for the synaptic order. If, in addition, $A$ is Banach, then we prove in Theorem \ref {th:Banachetc.} that the condition in question holds for the spectral order.

The dimension theory for synaptic algebras (not necessarily norm complete) with a complete projection lattice $P$ has been developed in \cite{FPtd, FPsym}. Recall that an element $s\in A$ is a \emph{symmetry} iff $s^2=1$. A symmetry $s$ is
said to \emph{exchange projections $p$ and $q$} iff $sps=q$. Let $A$ be an SA in which the projections form a complete OML $P$. The \emph{dimension equivalence} $\sim$ on $P$ is defined as follows: for $p,q \in P$, $p\sim q$ iff there are projections $p_1,\ldots,p_n$, $n\in\Nat$, $p=p_1, q=p_n$ such that $p_i$ and
$p_{i+1}$, $i=1,\ldots, n-1$ are exchanged by a symmetry. If $p$ and $q$ are exchanged by a symmetry, then
they are perspective (i.e., they share a common complement) in the OML $P$ \cite[Theorem 5.11]{FPsym}. Two
projections $p$ and $q$ are in \emph{position p$\sp{\prime}$} iff $p\wedge(1-q)=0=(1-p)\wedge q$. Two
projections $p,q$ in position p$\sp{\prime}$ are exchanged by a symmetry (whence they are perspective),
and therefore $p\sim q$ \cite[Corollary 3.9]{Pid}. A projection $p\in P$ is \emph{finite} iff, for any
$q\in P$, $q\sim p$ and $q\leq p$ imply $q=p$. The SA $A$ is of  \emph{finite type} iff every projection
in $P$ is finite.

\begin{lemma}\label{le:finmod} If $A$ is finite type, then  $P$ is a modular OML.
\end{lemma}
\begin{proof} Let $e,f,g$ be projections with $e\leq g$ and let $h:=(e\vee f)\wedge g$, $k:=e\vee(f\wedge g)$.
Then $k\leq h$, and both $h$ and $k$ have the property that their join with $f$ is $e\vee f$ and their meet
with $f$ is $f\wedge g$. Thus $h-(f\wedge g)$ and $(e\vee f)-f$ are in position p$\sp{\prime}$, and so are
$(e\vee f)-f$ and $k-(f\wedge g)$. It follows that $k-(f\wedge g) \sim h-(f\wedge g)$, and by finiteness, $h=k$.
\end{proof}

Since the Brouwer complementation defines a mapping $^{\sim}: E \to P$, before considering De Morgan properties,
we show some relations between $E$ and $P$.

\begin{lemma} \label{le:ForInduction}
Let $n\in\{0,1,2,...\}$, let $b\in A$ with $0\leq b\leq\frac{1}{2\sp{n}}$, and put $\lambda:=\frac{1}{2\sp{n+1}}$.
Then there exists $p\in P\cap CC(b)$ such that $0\leq b-\lambda p\leq\lambda$.
\end{lemma}

\begin{proof}
Put $p=1-p\sb{b,\lambda}\in P\cap CC(b)$. By Theorem \ref{th:specresprops} (i), $0\leq p(b-\lambda)$, i.e.,
$\lambda p\leq pb$. Since $0\leq b\leq\frac{1}{2\sp{n}}\leq 1$, $0\leq p\leq 1$, and $pCb$, we have
\[
pb\leq\frac{1}{2\sp{n}}p,\ pb\leq p, \text{\ and\ }pb\leq b.
\]
Thus, $\lambda p\leq pb\leq b$, i.e., $0\leq b-\lambda p$. Also by Theorem \ref{th:specresprops}
(i), $(1-p)(b-\lambda)\leq 0$, whence
\[
b-\lambda\leq p(b-\lambda)=pb-\lambda p\leq\frac{1}{2\sp{n}}p-\frac{1}{2\sp{n+1}}p=\frac{1}{2\sp{n+1}}p
=\lambda p,
\]
and therefore $b-\lambda p\leq\lambda$.
\end{proof}

The following theorem generalizes \cite[Theorem 4.1.13]{Mur} to an arbitrary SA.

\begin{theorem} \label{th:Mur}
If $e\in E$, there is a sequence of projections $(p\sb{n})\sb{n\in\Nat}\in CC(e)$ such that,
for every $n\in\Nat$, $0\leq e-\sum_{j=1}^n\frac{1}{2\sp{j}}p_j\leq\frac{1}{2\sp{n}}$.
Therefore, $e=\sum\sb{n=1}\sp{\infty}\frac{1}{2\sp{n}}p\sb{n}$, where the series converges
to $e$ in norm.
\end{theorem}

\begin{proof}
We construct by induction a sequence of projections $(p_n)\sb{n\in\Nat}\in CC(e)$ such that
$0\leq e-\sum_{j=1}^n\frac{1}{2\sp{j}}p_j\leq\frac{1}{2\sp{n}}$. Putting $n=0$, $b=e$, and
$p=p\sb{1}$ in Lemma \ref{le:ForInduction}, we obtain $p\sb{1}\in P\cap CC(a)$ with
$0\leq e-\frac12p\sb{1}\leq\frac12$. As our induction hypothesis, suppose that there
exist $p\sb{1},p\sb{2},...,p\sb{n}\in P\cap CC(e)$ such that $0\leq e-\sum_{j=1}^n\frac{1}
{2\sp{j}}p_j\leq\frac{1}{2\sp{n}}$. Now, putting $b=e-\sum_{j=1}^n\frac{1}{2\sp{j}}p_j$, and
$p=p\sb{n+1}$ in Lemma \ref{le:ForInduction}, we have $0\leq b\leq\frac{1}{2\sp{n}}$, and we
obtain $p\sb{n+1}\in P\cap CC(b)$ with $0\leq b-\frac{1}{2\sp{n+1}}p\sb{n+1}\leq\frac{1}{n+1}$,
i.e., $0\leq e-\sum_{j=1}^{n+1}\frac{1}{2\sp{j}}p_j\leq\frac{1}{2\sp{n+1}}$. Clearly, $b\in
CC(e)$, whence $p\sb{n+1}\in CC(b)\subseteq CC(e)$.
\end{proof}

\begin{corollary} \label{co:Mur}
Let $e\in E$ and let $(p\sb{n})\sb{n\in\Nat}$ be the sequence of projections in Theorem
\ref{th:Mur}. Then, $e\dg=\bigvee\sb{n=1}\sp{\infty}p\sb{n}$.
\end{corollary}

\begin{proof}
For each $n\in\Nat$, let $e\sb{n}:=\sum\sb{j=1}\sp{n}\frac{1}{2\sp{j}}p\sb{j}$, noting that
$e\sb{n}\in E$. Thus, if $n\in\Nat$, then $\frac{1}{2\sp{n}}p\sb{n}\leq e\sb{n}\leq e\leq e\dg$,
so $p\sb{n}=(\frac{1}{2\sp{n}}p\sb{n})\dg\leq e\dg$. Therefore $e\dg$ is an upper bound in $P$
for $\{p\sb{n}:n\in\Nat\}$.

Now, suppose that $q\in P$ and $q$ is an upper bound for $\{p\sb{n}:n\in\Nat\}$. Then, for any
$j\in\Nat$, $p\sb{j}\leq q$, so $p\sb{j}=p\sb{j}q$, and it follows that $e\sb{n}q=e\sb{n}$, whence
$e\sb{n}\leq q$ for all $n\in\Nat$.. Therefore $0\leq(q-e\sb{n})$ and $q-e=\lim\sb{n\rightarrow
\infty}(q-e\sb{n})$, whence $0\leq q-e$ by \cite[Theorem 4.7]{Fsa}. Thus, $e\leq q$, so $e\dg\leq q$,
and it follows that $e\dg=\bigvee\sb{n=1}\sp{\infty}p\sb{n}$.
\end{proof}

\begin{theorem} \label{le:caha}
Let $A$ be an SA of finite type in which $P$ is a complete OML. Then, if $e,f\in E$ and the infimum
$e\wedge f$ with respect to the synaptic order exists in $E$, it follows that $(e\wedge f)\dg=
e\dg\wedge f\dg$.
\end{theorem}

\begin{proof} Owing to Lemma \ref{le:dMvee}, we only have to prove that $e^o\wedge f^o\leq
(e\wedge f)^o$. We follow \cite{CaHa}.

As in Theorem \ref{th:Mur}, $e=\sum_{i=1}^{\infty}\frac{1}{2\sp{i}}p_i$ and $f=\sum_{j=1}^{\infty}
\frac{1}{2\sp{j}}q_j$, whence for any positive integer $k$ we have
\[
e\geq\sum_{i=1}^k\frac{1}{2^i}p_i\geq\frac{1}{2^k}(p_1+\cdots+p_k)\geq\frac{1}{2^k}\bigvee_{i=1}^k p_i.
\]
and similarly, for any positive integer $\ell$,
\[
f\geq\frac{1}{2^\ell}\bigvee_{j=1}^\ell q_j.
\]
It follows that
\[
e,f\geq\lambda(\bigvee_{i=1}^k p_i\wedge\bigvee_{j=1}^\ell q_j),
\]
where $\lambda=\min(\frac{1}{2^k}, \frac{1}{2^\ell})$. Hence,
\[
(e\wedge f)^o\geq(\lambda(\bigvee_{i=1}^k p_i \wedge \bigvee_{i=1}^\ell q_j))^o=
\bigvee_{i=1}^k p_i \wedge \bigvee_{i=1}^\ell q_j.
\]

Since $A$ is finite type, it follows that $P$ is modular (Lemma \ref{le:finmod}), hence, since $P$ is
complete, it is a continuous geometry \cite{Kap}. This means that if $v,w\in P$ and $(v\sb{n})\sb{n\in\Nat}$
is an increasing sequence in $P$ with $v\sb{n}\nearrow v$ (i.e., $\bigvee\sb{n=1}\sp{\infty}v\sb{n}=v$),
then $v\sb{n}\wedge w\nearrow v\wedge w$. Taking this into account and taking limits first as $k\to\infty$
and then as $\ell\to\infty$ in the inequality $(e\wedge f)\dg\geq\bigvee_{i=1}^k p_i \wedge \bigvee_{i=1}
^\ell q_j$, we deduce that
\[
(e\wedge f)\dg\geq\bigvee_{i=1}^{\infty} p_i\wedge \bigvee_{j=1}^{\infty} q_j.
\]
By Corollary \ref{co:Mur}, $e\dg=\bigvee_{i=1}^{\infty} p_i$ and $f\dg=\bigvee_{j=1}^{\infty} q_j$,
and we obtain
\[
(e\wedge f)\dg\geq e\dg\wedge f\dg. \qedhere
\]
\end{proof}

\begin{theorem} \label{th:Banachetc.}
Let $A$ be a Banach SA of finite type. Then, for all $e,f\in E$,
$(e\wedge\sb{s}f)\dg=e\dg\wedge f\dg$.
\end{theorem}

\begin{proof}
By our hypotheses, $P$ is a continuous geometry, hence if $p,q\in P$ and $(p\sb{\alpha})\sb{\alpha\in\reals}$
is a decreasing family in $P$ with $p\sb{\alpha}\searrow p$, (i.e., $\bigwedge\sb{\alpha\in\reals}p\sb{\alpha}
=p$), then $p\sb{\alpha}\vee q\searrow p\vee q$.

As per Definition \ref{de:supinf} (ii) and Theorem \ref{th:minmax}, the spectral resolution of $e\wedge\sb{s}f$
is given by
\[
p\sb{e\wedge\sb{s}f,\lambda}=\bigwedge\sb{\lambda<\mu\in\reals}(p\sb{e,\mu}\vee p\sb{f,\mu}),\ \lambda\in\reals.
\]
From $p\sb{e,\lambda}\vee p\sb{f,\lambda}\leq p\sb{e,\mu}\vee p\sb{f,\mu}\leq p\sb{e,\nu}\vee p\sb{f,\nu}$ for
all $\lambda\leq\mu\leq\nu$ in $\reals$, we obtain
\[
p\sb{e,\lambda}\vee p\sb{f,\lambda}\leq \bigwedge\sb{\mu>\lambda}(p\sb{e,\mu}\vee p\sb{f,\mu})\leq
\bigwedge\sb{\nu>\mu}\bigwedge\sb{\mu>\lambda}(p\sb{e,\mu}\vee p\sb{f,\nu})
\]
\[
=\bigwedge\sb{\nu>\lambda}(p\sb{e,\lambda}\vee p\sb{f,\nu})=p\sb{e,\lambda}\vee p\sb{f,\lambda},
\]
and it follows that $p\sb{e\wedge\sb{s}f,\lambda}=p\sb{e,\lambda}\vee p\sb{f,\lambda}$. Therefore,
by Lemma \ref{le:SpecResCarrier},
\[
(e\wedge\sb{s}f)\dg=1-p\sb{e\wedge\sb{s}f,0}=(1-p\sb{e,0})\wedge(1-p\sb{f,0})=e\dg\wedge f\dg. \qedhere
\]
\end{proof}

We note that in \cite{CaHa, H, dG} it was shown that De\,Morgan laws are satisfied iff the corresponding
algebra is of finite type. For a Banach SA, we proved only sufficiency of the latter condition (in both
the synaptic and spectral order)---necessity remains an open problem.

\end{document}